\documentclass[11pt]{amsart}
\usepackage{latexsym}
\usepackage{graphicx}
\usepackage{subfigure}
\usepackage{amsmath}
\usepackage{amssymb}
\usepackage{mathrsfs}
\usepackage{epstopdf}
\usepackage[all]{xypic}
\newtheorem{theorem}{Theorem}[section]
\newtheorem{lemma}[theorem]{Lemma}
\newtheorem{prop}[theorem]{Proposition}

\newtheorem{definition}[theorem]{Definition}

\newenvironment{remark}{\noindent \textbf{Remark}.}{\hfill $\square$}

\numberwithin{equation}{section}

\title{ Deformations of pairs $(X,L)$ when $X$ is singular}
\author{Jie Wang}
\thanks{ Department of Mathematics, Ohio State University, Columbus, OH, 43210 (jwang@math.ohio-state.edu).}
\date{\today}
\begin{document}
\maketitle
\begin{abstract} We give an elementary construction of the tangent-obstruction theory of the deformations of the pair $(X,L)$ with $X$ a reduced local complete intersection scheme and $L$ a line bundle on $X$. This generalizes the classical deformation theory of pairs in case $X$ is smooth. A criteria for sections of $L$ to extend is also given.

\end{abstract}
\vspace{.5cm}
\section{introduction.}

 Throughout this paper, we will work over the complex numbers $\mathbb{C}$. The deformation theory of the pair $(X,L)$ for $X$ a smooth variety and $L$ a line bundle on $X$ was first used to study Petri's conjecture by Arbarello and Cornalba in \cite{AC}. It was proved there that first-order deformations of the pair $(X,L)$ are in natural one to one correspondence with 
$$\xi\in H^1(X,\mathcal{D}_1(L)),$$ 
where $\mathcal{D}_1(L)$ is the sheaf of holomorphic first-order differential operators, and $H^2(X,\mathcal{D}_1(L))$ is an obstruction space. Given a first-order deformation $\phi\in H^1(X,T_X)$ of $X$, there is a first-order deformation of $L$ along $\phi$ if and only if $\phi\cup c(L)=0\in H^2(X,\mathcal{O}_X)$, where $c(L)\in H^1(X,\Omega^1_X)$ is the first Chern class of $L$ in the sense of Atiyah.

Moreover, there is a natural differentiation map
\begin{eqnarray}\label{np}
\xymatrix{H^1(X,\mathcal{D}_1(L))\ar[r]^-{M}& Hom(H^0(X,L),H^1(X,L))}
\end{eqnarray}
such that a section $s\in H^0(X,L)$ extends to first order along $\xi$ if and only if the element
$$M(\xi)(s)\in H^1(X,L)$$
is zero.

The map $M$ together with the tangent obstruction spaces have numerous deformation theoretic applications. For instance, for any first-order deformation of $(X,L)$, at least $h^0(L)-h^1(L)$ linearly independent sections of $L$ extend; $Ker(M)\subset H^1(X,\mathcal{D}_1(L))$ is the space of first-order deformations of $(X,L)$ to which all sections of $L$ extend.
If $X$ is a complete curve, a dual form of (\ref{np}) is the higher $\mu$-map $\mu_1$ in \cite{ACGH1}. In case $L$ gives an embedding of $X$ into some projective space $\mathbb{P}$, $Coker(M)$ is naturally isomorphic to $H^1(X,N_{X|\mathbb{P}})$ (cf. \cite{AC}), and therefore the surjectivity of $M$ implies that $X\subset\mathbb{P}$ is unobstructed.  Another direct consequence  is that the deformations of the pair $(X,L)$ is unobstructed for smooth curve $X$, since $H^2(X,\mathcal{D}_1(L))=0$. 
  If $X$ is a smooth $K3$-surface, the map $\xymatrix{H^1(T_X)\ar[r]^-{\cup c(L)}&H^2(\mathcal{O}_X)\cong\mathbb{C}}$ is surjective for every nontrival line bundle $L$. This means that $L$ deforms along a 19-dimensional subspace of $H^1(T_X)$, because $h^1(X,T_X)=20$. 

In this paper, we give an elementary approach to the deformation theory of the pair $(X,L)$ for $X$ a separated reduced local complete intersection scheme (l.c.i) of finite type over $\mathbb{C}$. We prove that even though $X$ could be singular, the functor of Artin rings
  $$Def_{(X,L)}(A)=\{\text{Flat deformations of } (X,L) \text{ over $A$}\}/\text{isomorphisms}$$
   still behaves well in the sense that there is a tangent-obstruction theory for this deformation functor, with tangent space $\mathrm{Ext}_{\mathcal{O}_X}^1(\mathcal{P}^1_X(L),L)$ and obstruction space $\mathrm{Ext}_{\mathcal{O}_X}^2(\mathcal{P}^1_X(L),L)$, where $\mathcal{P}^1_X(L)$ is the sheaf of one jets or sheaf of principle parts of $L$ on $X$. Moreover, there is a natural map analogous to $M$ characterizing obstructions for sections of $L$ to extend. Therefore, all the nice consequences mentioned above generalize to reduced l.c.i schemes.
If $X$ is smooth, $\mathcal{P}^1_X(L)={\mathcal{D}_1(L)}^{*}\otimes L$, where $\mathcal{D}_1(L)$ is the sheaf of first-order differential operators on $L$, and ${\mathrm{Ext}}_{\mathcal{O}_X}^i(\mathcal{P}^1_X(L),L)=H^i(X,\mathcal{D}_1(L))$. We go back to the classical case. The tangent and obstruction spaces for deformations of $(X,L)$ was known to experts and was stated implicitly in \cite{I}, \cite{I1}. Our approach is more elementary and does not use the more abstract machinery of cotangent complexes. It seems to the author that the criteria for sections of $L$ to extend is new. 

{\bf Acknowledgements.} The author would like to thank his advisor Herb Clemens for valuable suggestions and constant support, E. Sernesi and M. Manetti for pointing out a gap in the previous version of the paper.

\vspace{.5cm}
 \section{the sheaf of one jets}
In this section, we briefly review some basic facts and definitions about the sheaf of one jets.

 Let $g:X\rightarrow Y$ be a morphism between two algebraic schemes (separated schemes of finite type over $\mathbb{C}$), $L$ be a line bundle on $X$, and let $\Delta\subset X\times _Y X$ be the diagonal
defined by ideal sheaf  $\mathcal{I}_\Delta$. Consider the first
order neighborhood $Spec\frac{\mathcal{O}_{X\times_Y
X}}{{\mathcal{I}_\Delta}^2} $ of $\Delta$ with two
projections $\pi_1, \pi_2$ to $X$. The sheaf of one jets $\mathcal{P}^1_{X/Y}(L)$ of $X$ over $Y$
is defined to be $\mathcal{P}^1_{X/Y}(L):=\pi_{1*}\pi_2^*(L)$. $\mathcal{P}^1_{X/Y}(L)$ has a natural left $\mathcal{O}_X$-module structure induced by $\pi_1$ and a right $\mathcal{O}_X$-module structure induced by $\pi_2$ which, in general, is not equivalent to the left one. Throughout this paper, we will only use the left $\mathcal{O}_X$-module structure of $\mathcal{P}^1_{X/Y}(L)$. Consider the
short exact sequence 
$$0\longrightarrow\frac{\mathcal{I}_\Delta}{\mathcal{I}_\Delta^2}\longrightarrow\frac{\mathcal{O}_{X\times_Y X}}
{\mathcal{I}_\Delta^2}\longrightarrow\frac{\mathcal{O}_{X\times_Y
X}}{\mathcal{I}_\Delta}\longrightarrow0$$ 
Tensoring the above sequence with $\pi_2^*L$ then applying the functor
$\pi_{1*}$, we get a short exact sequence of left $\mathcal{O}_X$-modules on $X$

\begin{eqnarray}\label{1jet}\xymatrix{0\ar[r]&\Omega^1_{X/Y}(L)\ar[r]^i& \mathcal{P}^1_{X/Y}(L)\ar[r]&L\ar[r]&0}
\end{eqnarray}
where $\Omega^1_{X/Y}$ is the sheaf of relative K\"{a}hler differentials.
The sequence is exact on the right because there is no higher
derived image for $\pi_{1*}$ ($\pi_{1}$ has relative dimension $0$). When $Y=Spec(\mathbb{C})$, we will write $\mathcal{P}^1_X(L)$ for  $\mathcal{P}^1_{X/Y}(L)$.
The \textquotedblleft fibre\textquotedblright of the sheaf $\mathcal{P}^1_{X/Y}(L)$ at a closed point $x\in X$ is the stalk of $L|_{g^{-1}(g(x))}$  at $x$ mod the maximal ideal squared, i.e.
$$\mathcal{P}^1_{X/Y}(L)_x\otimes_{\mathcal{O}_{X,x}}\frac{\mathcal{O}_{X,x}}{m_x}\cong\frac{L_x}{({m_x}^2+m_{g(x)})L_x}.$$
This is the reason $\mathcal{P}^1_{X/Y}(L)$ is called the sheaf of (relative) one jets. 
There is a $\mathcal{O}_Y$-linear splitting $p_1: L\rightarrow\mathcal{P}^1_{X/Y}(L)$, which sends a section $s$ of $L$ to its one jet $\pi_{1*}\pi_2^*s$. $p_1$ satisfies the property that 
\begin{eqnarray} \label{pmap}
p_1(fs)=i(df\otimes s)+fp_1(s) 
\end{eqnarray}
for any$f\in\mathcal{O}_X(U)$ and $s\in L(U)$ where $U\subset X$ is any open subset.(In fact, $p_1$ is $\mathcal{O}_X$-linear if we use the right $\mathcal{O}_X$-module structure of $\mathcal{P}^1_{X/Y}(L)$). If $X$ is smooth, $Y=Spec(\mathbb{C})$, $\mathcal{P}^1_{X}(L)$ is the vector bundle  $\mathcal{H}om_{\mathcal{O}_X}(\mathcal{D}_1(L),L)$, where $\mathcal{D}_1(L)$ is the sheaf of first-order differential operators on $L$.

 \section{computation of the tangent space}
 In this section, let $X$ be a reduced algebraic scheme. Applying the functor
$\mathcal{H}om_{\mathcal{O}_X}(-,L)$ to (\ref{1jet}), we get a
long exact sequence
$$...\longrightarrow \mathrm{Ext}_{\mathcal{O}_X}^1(L,L)\longrightarrow
\mathrm{Ext}_{\mathcal{O}_X}^1(\mathcal{P}^1_X(L),L)\longrightarrow
\mathrm{Ext}_{\mathcal{O}_X}^1(\Omega^1_X(L),L)\longrightarrow$$ 
$$\ \ \ \ \ \ \ \longrightarrow \mathrm{Ext}_{\mathcal{O}_X}^2(L,L)\longrightarrow
\mathrm{Ext}_{\mathcal{O}_X}^2(\mathcal{P}^1_X(L),L)\longrightarrow
\mathrm{Ext}_{\mathcal{O}_X}^2(\Omega^1_X(L),L)\longrightarrow...$$ 

Notice that
$\mathrm{Ext}_{\mathcal{O}_X}^1(\Omega^1_X(L),L)=\mathrm{Ext}_{\mathcal{O}_X}^1(\Omega^1_X,\mathcal{O}_X)$
is the tangent space of the deformations of $X$, and
$\mathrm{Ext}_{\mathcal{O}_X}^1(L,L)=H^1(\mathcal{O}_X)$ is the tangent space of
deformations of $L$ with the base $X$ fixed. This suggests that
$\mathrm{Ext}_{\mathcal{O}_X}^1(\mathcal{P}^1_X(L),L)$ is the tangent space of
deformations of the pair $(X,L)$ and $\mathrm{Ext}_{\mathcal{O}_X}^2(\mathcal{P}^1_X(L),L)$ is an obstruction space. If $X$ is smooth,
$\mathcal{H}om_{\mathcal{O}_X}(\mathcal{P}^1_X(L),L)$ is the sheaf of first-order
differential operators $\mathcal{D}_1(L)$, and
$\mathrm{Ext}_{\mathcal{O}_X}^1(\mathcal{P}^1_X(L),L)=H^1(X,\mathcal{D}_1(L))$ is the
correct tangent space. In this section and the next, we will prove this is indeed the correct generalization of the tangent-obstruction theory for deformations of the pair $(X,L)$.

Let's first recall that for any reduced algebraic scheme over $\mathbb{C}$, we have an one-to-one correspondence between isomorphism classes of extensions of $X$ by a coherent locally free $\mathcal{O}_X$-module $\mathcal{I}$ and $\mathrm{Ext}^1_{\mathcal{O}_X}(\Omega^1_X,\mathcal{I})$ in the following way:

Given an isomorphism class of extension of $\mathcal{O}_X$ by $\mathcal{I}$,

$$\xymatrix{0\ar[r]&\mathcal{I}\ar[r]&\mathcal{O}_{\mathcal{X}}\ar[r]&\mathcal{O}_X\ar[r]&0},$$

when $X\subset \mathcal{X}$ is a closed immersion defined by ideal sheaf $\mathcal{I}$, and $\mathcal{I}^2=0$ in $\mathcal{O}_{\mathcal{X}}$,

 we associate to it (the isomorphism class of) the conormal sequence
$$\xymatrix{\mathcal{E}:0\ar[r]&\mathcal{I}\ar[r]&\Omega^1_{\mathcal{X}}|_X\ar[r]&\Omega^1_X\ar[r]&0}$$
(which is also exact on the left.)

 This conormal sequence corresponds to an element $c_{\mathcal{E}}$ in $\mathrm{Ext}^1_{\mathcal{O}_X}(\Omega^1_X,\mathcal{I})$.
 
 Conversely, for any $\mathcal{O}_X$-module extension
\begin{eqnarray}\label{ext}
\xymatrix{0\ar[r]&\mathcal{I}\ar[r]&\mathcal{E}\ar[r]^-{h}&\Omega^1_X\ar[r]&0},
\end{eqnarray}

let $d:\mathcal{O}_X\rightarrow\Omega^1_X$ be the canonical derivation. Let $\mathcal{O}=\mathcal{O}_X\times_{\Omega^1_X}{\mathcal{E}}$ be the fibre product sheaf: over an open subset $U\subset X$ we have $\mathcal{O}(U)=\{(f,a):h(a)=df\}$, with ring structure given by
$$(f,a)(f',a')=(ff',fa'+f'a)$$
We get a commutative diagram:

\hskip 2in \xymatrix{0\ar[r]&\mathcal{I}\ar[r]^j\ar@{=}[d]&\mathcal{O}\ar[r]\ar[d]^{d'}&\mathcal{O}_X\ar[r]\ar[d]^d&0\\
0\ar[r]&\mathcal{I}\ar[r]&\mathcal{E}\ar[r]&\Omega^1_X\ar[r]&0}

It is easy to check that $d': \mathcal{O}\rightarrow\mathcal{E}$ is a $\mathbb{C}$-derivation, thus factors through $\Omega^1_{\mathcal{O}}\otimes_{\mathcal{O}}\mathcal{O}_X$. Therefore $\Omega^1_{\mathcal{O}}\otimes_{\mathcal{O}}\mathcal{O}_X\cong\mathcal{E}$ by $5$-lemma and we recover $\mathcal{X}$ from (\ref{ext}).

In case $\mathcal{I}=\mathcal{O}_X$, we can give $\mathcal{O}_{\mathcal{X}}$ a $\mathbb{C}[\epsilon]$-module structure by sending $\epsilon$ to $j(1)\in\mathcal{O}_{\mathcal{X}}$. The fact that $\epsilon\mathcal{O}_{\mathcal{X}}\cong\mathcal{O}_X$ means that $\mathcal{X}$ is flat over $Spec(\mathbb{C}[\epsilon])$. Therefore $\mathcal{X}$ is a first-order infinitesimal deformation of $X$.

For the deformations of the pair $(X,L)$, we have the following result:

\begin{theorem}\label{theorem3.1}
Let $X$ be a reduced scheme of finite
type over $\mathbb{C}$, $L$ be a line bundle on $X$.
\begin{enumerate}
\item[(1)] The tangent space of the functor of Artin rings
$Def_{(X,L)}$ is canonically identified with
$\mathrm{Ext}_{\mathcal{O}_X}^1(\mathcal{P}^1_X(L),L)$.
\vspace{.3cm}
\item[(2)] There exists a natural pairing 
\begin{eqnarray}\label{natural pairing}\xymatrix{\mathrm{Ext}_{\mathcal{O}_X}^1(\mathcal{P}^1_X(L),L)\otimes H^0(X,L)\ar[r]^-{p}& H^1(X,L)}.
\end{eqnarray}
such that for any first-order deformation of the pair $(X,L)$ corresponding to $\xi\in{\mathrm{Ext}}_{\mathcal{O}_X}^1(\mathcal{P}^1_X(L),L)$, a
section $s\in H^0(L)$ extends to first order along $\xi$ if and only
if $\xi$ and $s$ pair to zero under $p$.

\end{enumerate}
\end{theorem}
\begin{proof}
\begin{enumerate}
\item[(1)] Given a first-order deformation of the pair $(X,L)$, i.e.
the following fibered diagram with $\mathcal{O}_{\mathcal{X}}$ flat over $Spec(\mathbb{C}[\epsilon])$ and $\mathcal{L}$ line bundle on $\mathcal{X}$:
\vspace{.2in}

\hskip 2in\xymatrix{L\ar@{^{(}->}[r] \ar[d] &\mathcal{L}\ar[d]\\
X\ar@{^{(}->}[r] \ar[d] & \mathcal{X}\ar[d]\\
Spec(\mathbb{C})\ar@{^{(}->}[r]& Spec(\mathbb{C}[\epsilon]) }
\vspace {.2in}

We have a diagram of (left) $\mathcal{O}_X$-modules:

\vspace{.2in}
\hskip 1in\xymatrix{&&0\ar[d]&0\ar[d]&\\
0\ar[r]&L\ar[r]\ar@{=}[d]&\Omega^1_{\mathcal{X}}(\mathcal{L})|_X\ar[r]\ar[d]&\Omega^1_X(L)\ar[r]\ar[d]^i&0\\
0\ar[r]&L\ar[r]&\mathcal{P}^1_{\mathcal{X}}(\mathcal{L})|_X\ar[r]^r \ar[d]&\mathcal{P}^1_X(L)\ar[r]\ar[d]&0\\
&&\mathcal{L}|_X\ar@{=}[r]\ar[d]&L\ar[d]\\
&&0&0&}

\vspace{.2in}
The two right columns are exact by (\ref{1jet}), and the fact that restriction to $X$ is (left) exact since $\mathcal{T}{or}^1_{\mathcal{O}_{\mathcal{X}}}(\mathcal{L},\mathcal{O}_X)=0$. ($\mathcal{L}$ is a locally free $\mathcal{O}_{\mathcal{X}}$-module!) The first row is the conormal sequence of $X\subset\mathcal{X}$ twisted by $L$, which is exact. Thus by Snake Lemma, $ker(r)=L$ and the second row is exact . Therefore, we can associate any first-order deformation of the pair $(X,L)$ the second row exact sequence, which corresponds to an element of $\mathrm{Ext}_{\mathcal{O}_X}^1(\mathcal{P}^1_X(L),L)$.

Now consider the commutative diagram
$$\xymatrix{0\ar[r]&L\ar[r]\ar@{=}[d]&\mathcal{L}\ar[r]\ar[d]^-{p_1'}&L\ar[r]\ar[d]^-{p_1}&0\\
0\ar[r]&L\ar[r]&\mathcal{P}^1_{\mathcal{X}}(\mathcal{L})|_X\ar[r]^-{r}&\mathcal{P}^1_X(L)\ar[r]&0}$$
where $p_1'$ is the composition of $p_1:\mathcal{L}\rightarrow\mathcal{P}^1_{\mathcal{X}}(\mathcal{L})$ and the restriction map to $X$. Thus $p_1'$ factors through $L\times_{\mathcal{P}^1_X(L)}\mathcal{P}^1_{\mathcal{X}}(\mathcal{L})|_X$ and therefore $\mathcal{L}\cong L\times_{\mathcal{P}^1_X(L)}\mathcal{P}^1_{\mathcal{X}}(\mathcal{L})|_X$. This fact suggests that we can recover $\mathcal{L}$ from  $\mathcal{P}^1_{\mathcal{X}}(\mathcal{L})|_X$ and $L$.
\vspace{.2in}

Conversely, for any element
$\xi\in\mathrm{Ext}_{\mathcal{O}_X}^1(\mathcal{P}^1_X(L),L)$ corresponding to an $\mathcal{O}_X$-module extension:
$$\xymatrix{0\ar[r]&L\ar[r]&\mathcal{E}\ar[r]^-{r}&
\mathcal{P}^1_X(L)\ar[r]&0.}$$ 
The pull back extension
$\mathcal{E}'=\mathcal{E}\times_{\mathcal{P}^1_X(L)}\Omega^1_X(L)$ by the natural inclusion
$$i:\Omega^1_X(L)\longrightarrow\mathcal{P}^1_X(L),$$
 sits naturally in the diagram

\hskip1in\xymatrix{0\ar[r]&L\ar[r]\ar@{=}[d]&\mathcal{E}'=\mathcal{E}\times_{\mathcal{P}^1_X(L)}\Omega^1_X(L)\ar[r]\ar[d]^{i'}&\Omega^1_X(L)\ar[r]\ar[d]^i&0\\
0\ar[r]&L\ar[r]&\mathcal{E}\ar[r]^r&\mathcal{P}^1_X(L)\ar[r]&0}

The first row exact sequence corresponds to an element in
$\mathrm{Ext}_{\mathcal{O}_X}^1(\Omega^1_X(L),L)=\mathrm{Ext}_{\mathcal{O}_X}^1(\Omega^1_X,\mathcal{O}_X)$,
which corresponds to a first-order infinitesimal deformation
$\mathcal{X}$ of $X$ as described in the beginning of this section.

To recover the deformation of $L$, let $\mathcal{E}''=\mathcal{E}'\otimes L^{-1}$ and let
$$\mathcal{L}=L\times_{\mathcal{P}^1_X(L)}\mathcal{E}=\{(s,e)\in L\oplus\mathcal{E}|\,  p_1(s)=r(e)\}.$$
$\mathcal{L}$ has natural $\mathcal{O}_X\times_{\Omega^1_X}\mathcal{E}''(=\mathcal{O}_{\mathcal{X}})$-module structure as below
$$(f,a)(s,e)=(fs,fe+i'(a\cdot s))$$
where $(f,a)\in\mathcal{O}_X\times_{\Omega^1_X}\mathcal{E}''=\mathcal{O}_{\mathcal{X}}$,  $(s,e)\in\mathcal{L}$ and $a\cdot s\in\mathcal{E}'$. This is a well defined $\mathcal{O}_{\mathcal{X}}$-module because 
$$p_1(fs)=i(df\otimes s)+fp_1(s)=r(i'(a\cdot s))+fr(e).$$ 

In order to see $\mathcal{L}$ is a locally free $\mathcal{O}_{\mathcal{X}}$-module of rank one, it suffices to prove the case $L$ is the trivial bundle since the question is local. In this case, (\ref{1jet}) splits (as left $\mathcal{O}_X$-module) and $\mathcal{P}^1_X(\mathcal{O}_X)\cong\mathcal{O}_X\oplus\Omega^1_X$. The statement follows immediately from this.
\vspace{.2in}
\item[(2)]
For any $\xi\in\mathrm{Ext}_{\mathcal{O}_X}^1(\mathcal{P}^1_X(L),L)$ corresponding to the extension
$$0\longrightarrow L\longrightarrow\mathcal{P}^1_{\mathcal{X}}(\mathcal{L})|_X\longrightarrow\mathcal{P}^1_X(L)\longrightarrow0.$$
Define the natural pairing $p(\xi\otimes s):=\delta(p_1(s))\in H^1(L)$. Where $\delta: H^0(\mathcal{P}^1_X(L))\rightarrow H^1(L)$ is the connecting homomorphism of the long exact cohomology sequence corresponding to $\xi$:

\hskip 1in\xymatrix{...\ar[r]&H^0(\mathcal{P}^1_{\mathcal{X}}(\mathcal{L})|_X)\ar[r]^-{r}&H^0(\mathcal{P}^1_X(L))\ar[r]^-{\delta}&H^1(L)\ar[r]&...}

$\delta(p_1(s))=0$ means there exists some $e\in H^0(\mathcal{P}^1_{\mathcal{X}}(\mathcal{L})|_X)$ such that $r(e)=p_1(s)$,
thus $(s,e)$ determines a global section of $\mathcal{L}=L\times_{\mathcal{P}^1_X(L)}\mathcal{P}^1_{\mathcal{X}}(\mathcal{L})|_X$.
\end{enumerate}
\end{proof}

\section{obstructions}
In this section, let $X$ be as in section 3 and we assume furthermore that $X$ is a local complete intersection scheme. We will show that $\mathrm{Ext}^2_{\mathcal{O}_X}( \mathcal{P}^1_X(L),L)$ is an obstruction space for deformations of the pair $(X,L)$.

The general idea is to apply Vistoli's construction of obstruction spaces for deformations of l.c.i schemes (cf. sections 3, 4 of \cite{V}) to the total space of ${L}^{\lor}$ and keep track of the bundle structure using a $\mathbb{C}^*$-action.

For any $z\in\mathbb{C}^*$, denote $\phi_z: L^{\lor}\rightarrow L^{\lor}$ be the multiplication map by $z$ in the fiber direction. Define a $\mathbb{C}^*$-action on $\mathcal{O}_{L^{\lor}}$ and  $\Omega_{L^{\lor}}^1$ by 
\begin{eqnarray}
z\cdot f=z^{-1}\phi_z^*f\\
z\cdot\omega=z^{-1}\phi_z^*\omega
\end{eqnarray}
for local sections $f\in\mathcal{O}_{L^{\lor}}$, $\omega\in\Omega^1_{L^{\lor}}$.

Let $\mathcal{O}_{L^{\lor}}^{\mathbb{C}^*}$ and $\Omega_{L^{\lor}}^{\mathbb{C}^*}$ be the sheaf of sections which are invariant under the $\mathbb{C}^*$-action. Both $\mathcal{O}_{L^{\lor}}^{\mathbb{C}^*}$ and $\Omega_{L^{\lor}}^{\mathbb{C}^*}$ have natural $\mathcal{O}_X$-module structures. Under some trivialization of $L^{\lor}$ over $U\subset X$: $L^{\lor}_U\cong U\times \mathbb{A}^1_t$, $\mathcal{O}_{L^{\lor}}^{\mathbb{C}^*}$ consists of functions on $L^{\lor}$ of the form $f(x)t$, and $\Omega_{L^{\lor}}^{\mathbb{C}^*}$ consists of $1$-forms $f(x)d_{L^{\lor}}t+\omega(x)t$ where $f$ is the pull back of a function on $U$ and $\omega\in\Omega^1_U$.

We have natural isomorphisms of $\mathcal{O}_X$-modules $\mathcal{O}_{L^{\lor}}^{\mathbb{C}^*}\cong L$ and $\mathcal{P}^1_X(L)\cong\Omega_{L^{\lor}}^{\mathbb{C}^*}$. The isomorphisms can be described as follows: for any section $s\in L$, we can naturally view it as a function on the total space of $L^{\lor}$ which restricts to a linear function on the fiber. Such functions are invariant under the $\mathbb{C}^*$-action and vice versa. This gives the first isomorphism. The second isomorphism is the natural one which identifies $p_1(s)$ with $d_{L^{\lor}}(f_s)$, where s is any section of $L$, $f_s$ is the function on $L^{\lor}$ corresponding to $s$ and $d_{L^{\lor}}$ is the exterior derivative on $L^{\lor}$. Under some local trivialization of $L^{\lor}$, it sends $(f,\omega)\in\mathcal{P}^1_X(L)$ to $f(x)d_{L^{\lor}}t+\omega(x)t\in\Omega_{L^{\lor}}^{\mathbb{C}^*}$.

Let 
$$e:\xymatrix{0\ar[r]&J\ar[r]&\widetilde{A}\ar[r]&A\ar[r]&0}$$ 
be a small extension of local artinian $\mathbb{C}$-algebras with $m_{\widetilde{A}} \cdot J =0$. Suppose we have a flat deformation $(\mathcal{X},\mathcal{L})$ of the pair $(X,L)$ over $Spec(A)$:

$$\xymatrix{L\ar[r]\ar[d]&\mathcal{L}\ar[d]\\ X\ar[r]\ar[d]&\mathcal{X}\ar[d]^f\\Spec(\mathbb{C})\ar[r]&Spec(A)}$$

Let $(\widetilde{\mathcal{X}}_{\alpha},\widetilde{\mathcal{L}}_{\alpha})$ and $(\widetilde{\mathcal{X}}_{\beta},\widetilde{\mathcal{L}}_{\beta})$ be two liftings of $(\mathcal{X},\mathcal{L})$ to $Spec(\widetilde{A})$. We would like to measure the difference of two such liftings. 

Let's restrict ourselves to the local situation first. Suppose that $\mathcal{X}$ is affine, embedded in $\mathcal{S}=Spec( A[x_1,...,x_n])$ and the total space of $\widetilde{\mathcal{L}}^{\lor}_i$ are both embedded into $Spec(\widetilde{A}[x_1,...,x_n])\times\mathbb{A}^1=\widetilde{\mathcal{S}}\times\mathbb{A}^1$ with image $\widetilde{\mathcal{X}}_i\times\mathbb{A}^1$. 

Let $\mathcal{I}_0$ be the ideal sheaf of $L^{\lor}$ in $S\times\mathbb{A}^1$. The conormal sequence 
$$\xymatrix{0\ar[r]&\frac{\mathcal{I}_0}{\mathcal{I}_0^2}\ar[r]^-{d}&\Omega_{{S}\times\mathbb{A}^1}|_{L^{\lor}}\ar[r]&\Omega_{L^{\lor}}\ar[r]&0}$$
is exact because $L^{\lor}$ is l.c.i.
Taking the invariant part under the $\mathbb{C}^*$-action we get an exact sequence of $\mathcal{O}_{X}$-modules

\begin{eqnarray}\label{invariant}
\xymatrix{0\ar[r]&{(\frac{\mathcal{I}_0}{\mathcal{I}_0^2})}^{\mathbb{C}^*}\ar[r]^-{d'}&\Omega^{\mathbb{C}^*}_{{S}\times\mathbb{A}^1}|_{L^{\lor}}\ar[r]&\Omega^{\mathbb{C}^*}_{L^{\lor}}\ar[r]&0}.
\end{eqnarray}

The difference of $\widetilde{\mathcal{L}}^{\lor}_{\alpha}$ and $\widetilde{\mathcal{L}}^{\lor}_{\beta}$ as embedded deformations corresponds to an $\mathcal{O}_{L^{\lor}}$-module homomorphism $v_{\alpha\beta}: \frac{\mathcal{I}_0}{\mathcal{I}_0^2}\rightarrow J\otimes_{\mathbb{C}}\mathcal{O}_{L^{\lor}}$. The fact that $\widetilde{\mathcal{L}}^{\lor}_i$ is embedded as $\widetilde{\mathcal{X}}_i\times\mathbb{A}^1$ implies that $v_{\alpha\beta}$ sends the invariant part ${(\frac{\mathcal{I}_0}{\mathcal{I}_0^2})}^{\mathbb{C}^*}$ to the invariant part $J\otimes_{\mathbb{C}}\mathcal{O}^{\mathbb{C}^*}_{L^{\lor}}=J\otimes_{\mathbb{C}}L$. Denote the restriction $v_{\alpha\beta}'$. 

Now, take the push-out of (\ref{invariant}) under $v'_{\alpha\beta}$, we obtain an $\mathcal{O}_X$-module extension $\mathcal{E}_{\alpha\beta} $ of $\mathcal{P}^1_{X}(L)$ by $J\otimes_{\mathbb{C}}L$:
\begin{eqnarray}\label{split}
\xymatrix{0\ar[r]&{(\frac{\mathcal{I}_0}{\mathcal{I}_0^2})}^{\mathbb{C}^*}\ar[r]^-{d'}\ar[d]^-{v'_{\alpha\beta}}&\Omega^{\mathbb{C}^*}_{{S}\times\mathbb{A}^1}|_{L^{\lor}}\ar[r]\ar[d]^-{\psi_{\alpha\beta}}&\Omega^{\mathbb{C}^*}_{L^{\lor}}\ar[r]\ar@{=}[d]&0\\
0\ar[r]&J\otimes_{\mathbb{C}}L\ar[r]^{l_{\alpha\beta}}&\mathcal{E}_{\alpha\beta}\ar[r]&\mathcal{P}^1_{X}(L)\ar[r]&0}
\end{eqnarray}

\begin{lemma}\label{cano}
The extension $\mathcal{E}_{\alpha\beta}$ does not depend on the choice of $\widetilde{\mathcal{S}}$.
\end{lemma}
\begin{proof}
Suppose there are two embeddings $\widetilde{\mathcal{L}}^{\lor}_i\rightarrow\widetilde{\mathcal{S}}_{j}\times\mathbb{A}^1$, where $i=\alpha,\beta$, $j=1,2$; reducing to embeddings 
$\mathcal{L}^{\lor}\rightarrow\mathcal{S}_{1}\times\mathbb{A}^1$ and $\mathcal{L}^{\lor}\rightarrow\mathcal{S}_{2}\times\mathbb{A}^1$. These induce embeddings 
$$\widetilde{\mathcal{L}}^{\lor}_i\rightarrow\widetilde{\mathcal{S}}_1\times_{Spec(\widetilde{A})}\widetilde{\mathcal{S}}_2\times\mathbb{A}^1$$
 reducing to $$\mathcal{L}^{\lor}\rightarrow\mathcal{S}_1\times_{Spec(A)}\mathcal{S}_2\times\mathbb{A}^1.$$
  Let $C_1$, $C_2$, $C_{12}$ be the conormal bundles of $L^{\lor}$ in $S_{1}\times\mathbb{A}^1$, $S_{2}\times\mathbb{A}^1$, $S_{1}\times S_{2}\times\mathbb{A}^1$ respectively. Denote by $v'_{j}: C_{j}^{\mathbb{C}^*}\rightarrow J\otimes_{\mathbb{C}}L$ the invariant part of the corresponding sections of the normal bundles, $\mathcal{E}_j=v_{j*}'\Omega^{\mathbb{C}^*}_{{S_j}\times\mathbb{A}^1}|_{L^{\lor}}$,  $\mathcal{E}_{12}=v_{12*}'\Omega^{\mathbb{C}^*}_{{S_1}\times S_2\times\mathbb{A}^1}|_{L^{\lor}}$ and $p_{j}:C_{j}^{\mathbb{C}^*}\rightarrow C^{\mathbb{C}^*}_{12}$ be the natural map between conormal bundles. Then 
  $$v'_{12}\circ p_j=v'_j: C^{\mathbb{C}^*}_j\rightarrow J\otimes_{\mathbb{C}}L.$$
   We have the following diagram
$$\xymatrix{0\ar[r]&C_j^{\mathbb{C}^*}\ar[r]\ar[d]^-{p_j}&\Omega^{\mathbb{C}^*}_{{S_j}\times\mathbb{A}^1}|_{L^{\lor}}\ar[r]\ar[d]&\Omega^{\mathbb{C}^*}_{L^{\lor}}\ar[r]\ar@{=}[d]&0\\
0\ar[r]&C^{\mathbb{C}^*}_{12}\ar[r]\ar[d]^-{v'_{12}}&\Omega^{\mathbb{C}^*}_{S_1\times S_2\times\mathbb{A}^1}|_{L^{\lor}}\ar[r]\ar[d]&\Omega^{\mathbb{C}^*}_{L^{\lor}}\ar[r]\ar@{=}[d]&0\\
0\ar[r]&J\otimes_{\mathbb{C}}L\ar[r]&\mathcal{E}_{12}\ar[r]&\Omega^{\mathbb{C}^*}_{L^{\lor}}\ar[r]&0}$$
By the universal property of push out, this diagram induces isomorphism of extensions $\psi_j:\mathcal{E}_j\cong\mathcal{E}_{12}$. We define the canonical isomorphism between $\mathcal{E}_2$ and $\mathcal{E}_1$ to be $\psi_1\circ\psi_2^{-1}$.
\end{proof}

\begin{prop}\label{prop4.2}
For any two liftings of line bundles $\widetilde{\mathcal{L}}^{\lor}_{\alpha}$, $\widetilde{\mathcal{L}}^{\lor}_{\beta}$ inside $\widetilde{\mathcal{S}}\times\mathbb{A}^1$ as above, there is an $\mathcal{O}_X$-module extension $\mathcal{E}_{\alpha\beta}$ of $\mathcal{P}^1_X(L)$ by $J\otimes_{\mathbb{C}}L$, well defined up to canonical isomorphism, with the following properties.
\begin{enumerate}
\item For any three liftings $\widetilde{\mathcal{L}}^{\lor}_{\alpha}$, $\widetilde{\mathcal{L}}^{\lor}_{\beta}$, and $\widetilde{\mathcal{L}}^{\lor}_{\gamma}$, there is a canonical isomorphism of extensions
$$F_{\alpha\beta\gamma}: \mathcal{E}_{\alpha\beta}+\mathcal{E}_{\beta\gamma}\cong\mathcal{E}_{\alpha\gamma}\footnote{The sum of two extensions of $\mathcal{O}_X$-module $\xymatrix{0\ar[r]&\mathcal{G}\ar[r]^-{l_i}&\mathcal{E}_i\ar[r]^-{k_i}\ar[r]&\mathcal{F}\ar[r]&0}$ is defined to be the quotient of the submodule $\mathcal{B}=\{(e_1,e_2)\in\mathcal{E}_1\oplus\mathcal{E}_2:\ k_1(e_1)=k_2(e_2)\}$ by sections of the form $(l_1(y),-l_2(y))$, $y\in\mathcal{G}$.

The oposite extension $-\mathcal{E}$ is defined to be $\xymatrix{0\ar[r]&\mathcal{G}\ar[r]^-{-l}&\mathcal{E}\ar[r]^-{k}\ar[r]&\mathcal{F}\ar[r]&0}$.}$$ 
such that for any four liftings,
\begin{eqnarray}\label{compatible}
F_{\alpha\gamma\delta}\circ(F_{\alpha\beta\gamma}+id_{\mathcal{E}_{\beta\gamma}})=F_{\alpha\beta\delta}\circ(id_{\mathcal{E}_{\alpha\beta}}+F_{\beta\gamma\delta})
\end{eqnarray}
as homomorphism of extensions  from $\mathcal{E}_{\alpha\beta}+\mathcal{E}_{\beta\gamma}+\mathcal{E}_{\gamma\delta}$ to $\mathcal{E}_{\alpha\delta}$.

\item Given an $\mathcal{O}_X$-module extension $\mathcal{E}$ of $\mathcal{P}^1_X(L)$ by $J\otimes_{\mathbb{C}}L$, and an lifting $\widetilde{\mathcal{L}}^{\lor}_{\alpha}$ of $\mathcal{L}^{\lor}$, there is an abstract lifting $\widetilde{\mathcal{L}}^{\lor}_{\beta}$ such that $\mathcal{E}_{\alpha\beta}$ is isomorphic to $\mathcal{E}$.

\item There is a natural bijection between bundle isomorphisms $\Phi:\widetilde{\mathcal{L}}^{\lor}_{\alpha}\cong\widetilde{\mathcal{L}}^{\lor}_{\beta}$ with splittings of $\mathcal{E}_{\alpha\beta}$.

\end{enumerate}

\end{prop}
\begin{proof}
\begin{enumerate}
\item As embedded deformations we certainly have $v_{\alpha\beta}'+v_{\beta\gamma}'=v_{\alpha\gamma}'$ as homomorphisms from $({\frac{\mathcal{I}_0}{\mathcal{I}_0^2})}^{\mathbb{C}^*}$ to $J\otimes_{\mathbb{C}}L$ Then $\mathcal{E}_{\alpha\beta}+\mathcal{E}_{\beta\gamma}$ fits into the diagram
$$\xymatrix{0\ar[r]&{(\frac{\mathcal{I}_0}{\mathcal{I}_0^2})}^{\mathbb{C}^*}\ar[r]^-{d'}\ar[d]^-{v'_{\alpha\beta}+v_{\beta\gamma}'}&\Omega^{\mathbb{C}^*}_{{S}\times\mathbb{A}^1}|_{L^{\lor}}\ar[r]\ar[d]^-{(\psi_{\alpha\beta},\psi_{\beta\gamma})}&\Omega^{\mathbb{C}^*}_{L^{\lor}}\ar[r]\ar@{=}[d]&0\\
0\ar[r]&J\otimes_{\mathbb{C}}L\ar[r]&\mathcal{E}_{\alpha\beta}+\mathcal{E}_{\beta\gamma}\ar[r]&\mathcal{P}^1_{X}(L)\ar[r]&0}$$
By the universal property of push-out, there is a unique isomorphism $F_{\alpha\beta\gamma}^{-1}:\mathcal{E}_{\alpha\gamma}\rightarrow\mathcal{E}_{\alpha\beta}+\mathcal{E}_{\beta\gamma}$ such that $(\psi_{\alpha\beta},\psi_{\beta\gamma})$ factors through $F^{-1}_{\alpha\beta\gamma}$. The compatibility condition (\ref{compatible}) follows from the universal property of push-out as well.
 \vspace{.2cm}
 
\item Applying the derived functor $\mathrm{Hom}_{\mathcal{O}_X}(-,J\otimes_{\mathbb{C}}L)$ to (\ref{invariant}), we obtain exact sequence
$$\xymatrix{\mathrm{Hom}_{\mathcal{O}_X}((\frac{\mathcal{I}_0}{\mathcal{I}_0^2})^{\mathbb{C}^*},J\otimes_{\mathbb{C}}L)\ar[r]& \mathrm{Ext}^1_{\mathcal{O}_X}(\mathcal{P}^1_X(L),J\otimes_{\mathbb{C}}L)\ar[r]&\mathrm{Ext}^1_{\mathcal{O}_X}(\Omega^{\mathbb{C}^*}_{{S}\times\mathbb{A}^1}|_{L^{\lor}},J\otimes_{\mathbb{C}}L)}$$
where the last term is zero because $X$ is affine and $\Omega^{\mathbb{C}^*}_{{S}\times\mathbb{A}^1}|_{L^{\lor}}$ is locally free. Thus for any $$\mathcal{E}\in\mathrm{Ext}^1_{\mathcal{O}_X}(\mathcal{P}^1_X(L),J\otimes_{\mathbb{C}}L),$$ 
there is $$v'\in\mathrm{Hom}_{\mathcal{O}_X}((\frac{\mathcal{I}_0}{\mathcal{I}_0^2})^{\mathbb{C}^*},J\otimes_{\mathbb{C}}L)$$
 such that $v'_*\Omega^{\mathbb{C}^*}_{{S}\times\mathbb{A}^1}|_{L^{\lor}}\cong\mathcal{E}$. $v'$ can be uniquely extended to a $\mathcal{O}_{L^{\lor}}$-module homomorphism $v: \frac{\mathcal{I}_0}{\mathcal{I}_0^2}\rightarrow J\otimes_{\mathbb{C}}\mathcal{O}_{L^{\lor}}$. Now choose $\widetilde{\mathcal{L}}^{\lor}_{\beta}\subset\widetilde{\mathcal{S}}\times\mathbb{A}^1$ such that the difference of $\widetilde{\mathcal{L}}_{\beta}^{\lor}$ and $\widetilde{\mathcal{L}}_{\alpha}^{\lor}$ as embedded deformations corresponds to $v$, then by construction $\mathcal{E}_{\alpha\beta}\cong\mathcal{E}$. 
\vspace{.1cm}

\item
First notice that by the construction of push-out, to give a splitting $s:\mathcal{E}_{\alpha\beta}\rightarrow J\otimes_{\mathbb{C}}L$ is equivalent to give a $\mathcal{O}_{X}$-module homomorphism $D: \Omega^{\mathbb{C}^*}_{{S}\times\mathbb{A}^1}|_{L^{\lor}}\rightarrow J\otimes_{\mathbb{C}}L$ such that $D\circ d'=v'_{\alpha\beta}$. 

Now let $\phi: \mathcal{O}_{\widetilde{\mathcal{L}}^{\lor}_{\alpha}}\cong\mathcal{O}_{\widetilde{\mathcal{L}}^{\lor}_{\beta}}$ be a bundle isomorphism inducing identity on $\mathcal{O}_{\mathcal{L}^{\lor}}$. Consider the two projections $\pi_i: \mathcal{O}_{\widetilde{\mathcal{S}}\times\mathbb{A}^1}\rightarrow\mathcal{O}_{\widetilde{\mathcal{L}}^{\lor}_i}$. The difference 
$$D=\pi_{\beta}-\phi\circ\pi_{\alpha}:\mathcal{O}_{\widetilde{\mathcal{S}}\times\mathbb{A}^1}\rightarrow\mathcal{O}_{\widetilde{\mathcal{L}}^{\lor}_{\beta}}$$
will have its image inside $J\mathcal{O}_{\widetilde{\mathcal{L}}^{\lor}_{\beta}}=J\otimes_{\mathbb{C}}\mathcal{O}_{L^{\lor}}$. It is easy to check that 
$$D\in Der_{\widetilde{A}}(\mathcal{O}_{\widetilde{\mathcal{S}}\times\mathbb{A}^1},J\otimes_{\mathbb{C}}\mathcal{O}_{L^{\lor}})=Der_{\mathbb{C}}(\mathcal{O}_{S\times\mathbb{A}^1}, J\otimes_{\mathbb{C}}\mathcal{O}_{L^{\lor}})\cong Hom_{\mathcal{O}_{L^{\lor}}}(\Omega_{S\times\mathbb{A}^1}|_{L^{\lor}}, J\otimes_{\mathbb{C}}\mathcal{O}_{L^{\lor}})$$ 

and $D\circ d=v_{\alpha\beta}$. The fact that $\phi$ is a bundle isomorphism implies that $D$ sends $\Omega^{\mathbb{C}^*}_{{S}\times\mathbb{A}^1}|_{L^{\lor}}$ to $J\otimes_{\mathbb{C}}\mathcal{O}^{\mathbb{C}^*}_{L^{\lor}}=J\otimes_{\mathbb{C}}L$. This gives a spliting of $\mathcal{E}_{\alpha\beta}$.

Conversely, any $\mathcal{O}_{X}$-module homomorphism 
$$D: \Omega^{\mathbb{C}^*}_{{S}\times\mathbb{A}^1}|_{L^{\lor}}\rightarrow J\otimes_{\mathbb{C}}\mathcal{O}^{\mathbb{C}^*}_{L^{\lor}}$$
 with $D\circ d'=v'_{\alpha\beta}$ can be extended uniquely to a $\mathcal{O}_{L^{\lor}}$-module homomorphism $D: \Omega_{S\times\mathbb{A}^1}|_{L^{\lor}}\rightarrow J\otimes_{\mathbb{C}}\mathcal{O}_{L^{\lor}}$ with $D\circ d=v_{\alpha\beta}$. One checks easily that 
 $$\pi_{\beta}-D$$ vanishes on the ideal sheaf of $\widetilde{\mathcal{L}}_{\alpha}^{\lor}$ thus factors through $\pi_{\alpha}$, and therefore we recover the bundle isomorphism $\phi$ from such $D$.
\end{enumerate}
\end{proof}

\begin{remark} Proposition \ref{prop4.2} still holds in the global case. Since the local extension does not depending on the choice of embeddings, one can construct a global extension for any two abstract liftings $\widetilde{\mathcal{L}}_2^{\lor}$ and $\widetilde{\mathcal{L}}_1^{\lor}$ by glueing together the local extensions using the canonical isomorphisms in lemma \ref{cano} on the overlap of two open affine subsets. One checks easily that the glued extension satisfies the properties in the proposition. We will not need the global case in the construction of the obstruction space.

\end{remark}

The rest of the proof is entirely based on the construction in \cite{V}. The idea is to use extension cocycles to measure the obstructions to patching together local liftings (which always exist since $X$ is l.c.i) coherently. 

Here we collect some useful results about extension cocycles and refer to \cite{V} for details.
\begin{definition} Let $\mathcal{F}$, $\mathcal{G}$ be sheaves of $\mathcal{O}_X$-modules, $\{U_{\alpha}\}$ be an open covering of $X$. An extension cocycle 
$$(\{\mathcal{E}_{\alpha\beta}\},\{F_{\alpha\beta\gamma}\})$$
of $\mathcal{F}$ by $\mathcal{G}$ on $\{U_{\alpha}\}$ is a collection of extensions $\{\mathcal{E}_{\alpha\beta}\}$ of $\mathcal{F}|_{U_{\alpha\beta}}$ by $\mathcal{G}|_{U_{\alpha\beta}}$, and isomorphisms
$$F_{\alpha\beta\gamma}: \mathcal{E}_{\alpha\beta}+\mathcal{E}_{\beta\gamma}\cong\mathcal{E}_{\alpha\gamma}$$
on $U_{\alpha\beta\gamma}$ satisfying the compatibility condition as in (\ref{compatible}).
\end{definition}

Two extension cocycles $(\{\mathcal{E}_{\alpha\beta}\},\{F_{\alpha\beta\gamma}\})$, $(\{\mathcal{E}'_{\alpha\beta}\},\{F'_{\alpha\beta\gamma}\})$ are isomorphic if there exist isomorphism of extensions
$$\phi_{\alpha\beta}:\mathcal{E}_{\alpha\beta}\cong\mathcal{E}'_{\alpha\beta}$$
such that
$$\phi_{\alpha\gamma}\circ F_{\alpha\beta\gamma}=F'_{\alpha\beta\gamma}\circ(\phi_{\alpha\beta}+\phi_{\beta\gamma}).$$

\begin{definition}
We say an extension cocycle is a boundary if it is isomorphic to 
$$\partial\{\mathcal{E}_{\alpha}\}=(\{\mathcal{E}_{\alpha}-\mathcal{E}_{\beta}\},F_{\alpha\beta\gamma})$$
for a collection of extensions $\{\mathcal{E}_{\alpha}\}$ of $\mathcal{F}|_{U_\alpha}$ by $\mathcal{G}|_{U_\alpha}$, where
$$F_{\alpha\beta\gamma}:\mathcal{E}_\alpha-\mathcal{E}_\beta+\mathcal{E}_\beta-\mathcal{E}_\gamma\longrightarrow\mathcal{E}_\alpha-\mathcal{E}_\gamma$$
is the obvious isomorphism.
\end{definition}

The set of isomorphism classes of extension cocycles form an abelian group, and the boundaries form a subgroup. The quotient group is called the group of extension classes, and is denoted by $\Xi_{\mathcal{O}_X}(U_\alpha;\mathcal{F},\mathcal{G})$. We refer to section $3$ in \cite{V} for the proofs of the above facts.

\begin{theorem} \label{thm4.5}
For $\{U_\alpha\}$ a good cover, there is canonical group isomorphism of $\Xi_{\mathcal{O}_X}(U_\alpha;\mathcal{F},\mathcal{G})$ with the kernel of the localization map $\mathrm{Ext}^2_{\mathcal{O}_X}(\mathcal{F},\mathcal{G})\rightarrow H^0(X, \mathcal{E}xt^2(\mathcal{F},\mathcal{G}))$.

\end{theorem}
\begin{proof}
See Theorem (3.13) of \cite{V}.
\end{proof}

To finish the proof, we cover $\mathcal{X}$ by open affine subscheme $\{\mathcal{U}_{\alpha}\}$ such that $\mathcal{L}_{\alpha}^{\lor}=\mathcal{L}^{\lor}|_{\mathcal{U}_{\alpha}}$ has a lifting $\widetilde{\mathcal{L}}_{\alpha}^{\lor}$ over $\widetilde{\mathcal{U}}_{\alpha}$. The difference of $\widetilde{\mathcal{L}}_{\alpha}^{\lor}$ and $\widetilde{\mathcal{L}}_{\beta}^{\lor}$ on the overlap corresponds to an extension $\mathcal{E}_{\alpha\beta}$ of $\mathcal{P}^1_{U_{\alpha\beta}}(L_{\alpha\beta})$ by $J\otimes_{\mathbb{C}}L_{\alpha\beta}$. For each triple $\alpha$, $\beta$, $\gamma$, consider the 
isomorphism
$$F_{\alpha\beta\gamma}: \mathcal{E}_{\alpha\beta}+\mathcal{E}_{\beta\gamma}\cong\mathcal{E}_{\alpha\gamma}$$
in proposition \ref{prop4.2} $(a)$.

Then $(\mathcal{E}_{\alpha\beta},F_{\alpha\beta\gamma})$ is an extension cocycle, which we will denote simply by $(\mathcal{E}_{\alpha\beta})$. If $\overline{\mathcal{L}}^{\lor}_{\alpha}$ is another collections of liftings, coresponding to another extension cocycle $(\mathcal{E}'_{\alpha\beta})$,
we get isomorphisms
$$\mathcal{E}_{\alpha\beta}\cong\mathcal{E}'_{\alpha\beta}+\mathcal{E}(\widetilde{\mathcal{L}}_\alpha^{\lor},\overline{\mathcal{L}}_\alpha^{\lor})-\mathcal{E}(\widetilde{\mathcal{L}}_\beta^{\lor},\overline{\mathcal{L}}_\beta^{\lor}).$$
by proposition \ref{prop4.2} $(a)$.
One checks that this is an isomorphism of extension cocycles. Thus the class of 
$$[\mathcal{E}_{\alpha\beta}]\in\Xi_{\mathcal{O}_X}(U_\alpha;\mathcal{P}^1_X(L),J\otimes_{\mathbb{C}}L)$$
is independent of the choice of local liftings.

A global lifting exists if and only if we can choose local liftings $\widetilde{\mathcal{L}}_{\alpha}^{\lor}$ and isomorphisms of line bundles $\phi_{\alpha\beta}:\widetilde{\mathcal{L}}_{\alpha}^{\lor}\rightarrow\widetilde{\mathcal{L}}_{\beta}^{\lor}$ satisfying the cocycle condition
$$\phi_{\alpha\beta}\circ\phi_{\beta\gamma}=\phi_{\alpha\gamma}.$$
By proposition \ref{prop4.2} $(c)$, to give $\phi_{\alpha\beta}$ is equivalent to assigning splittings for $\mathcal{E}_{\alpha\beta}$. It is easy to check that $\phi_{\alpha\beta}$ satisfies cocycle condition if and only if $(\mathcal{E}_{\alpha\beta})$ is isomorphic to the trivial extension cocycle.

Conversely, if the class 
$$[\mathcal{E}_{\alpha\beta}]\in\Xi_{\mathcal{O}_X}(U_\alpha;\mathcal{P}^1_X(L),J\otimes_{\mathbb{C}}L)$$
is zero, $(\mathcal{E}_{\alpha\beta})$ is isomorphic to a boundary $(\mathcal{E}_{\alpha}-\mathcal{E}_{\beta})$. By proposition \ref{prop4.2} $(b)$, we can choose local lifting $\overline{\mathcal{L}}_{\alpha}^{\lor}$ such that $\mathcal{E}(\widetilde{\mathcal{L}}_{\alpha}^{\lor},\overline{\mathcal{L}}_{\alpha}^{\lor})\cong\mathcal{E}_{\alpha}$. Then $\overline{\mathcal{L}}_{\alpha}^{\lor}$ will patch together to give a global lifting.

Combine the above discussion with theorem \ref{thm4.5} and the fact that $\mathcal{E}xt^2_{\mathcal{O}_X}(\mathcal{P}^1_X(L),L)=0$ (since (\ref{invariant}) is a locally free resolution of $\mathcal{P}^1_X(L)$), we get
\begin{theorem} Let $X$ be a l.c.i scheme, $L$ a line bundle on $X$. For any small extension
$$e:\xymatrix{0\ar[r]&J\ar[r]&\widetilde{A}\ar[r]&A\ar[r]&0}$$
and any deformation $(\mathcal{X},\mathcal{L})$ of $(X,L)$ over $A$,
\begin{enumerate}
\item There is an element 
$$\circ(e)\in J\otimes_{\mathbb{C}}\mathrm{Ext}^2_{\mathcal{O}_X}(\mathcal{P}^1_X(L),L),$$
such that $\circ(e)=0$ if and only if a lifting $(\widetilde{\mathcal{X}},\widetilde{\mathcal{L}})$ of $(\mathcal{X},\mathcal{L})$ to $\widetilde{A}$ exists.
\item If a lifting exists, the set of isomorphism classes of liftings is a principal homogeneous space  for the group
$$J\otimes_{\mathbb{C}}\mathrm{Ext}^1_{\mathcal{O}_X}(\mathcal{P}^1_X(L),L).$$
\end{enumerate}
\end{theorem}


\begin{thebibliography}{999}

\bibitem{AC} Arbarello, E, Cornalba, M.: Su una congettura di Petri. Comment. Math. Helvetici {\bf56}, 1-38, 1981.

\bibitem{ACGH} Arbarello, E., Cornalba, M., Griffiths, P., Harris, J.: Geometry of Algebraic Curves, Volume  I. Springer Grundlehren {\bf267}, 1985. 

\bibitem{ACGH1} Arbarello, E., Cornalba, M., Griffiths, P., Harris, J.: Special Divisors on Algebraic Curves. Regional Algebraic Geometry Conference. Athens, Georgia, May, 1979.



\bibitem{EH1} Eisenbud, D., Harris, J.: Limit linear series: Basic theory. Invent. Math. {\bf85}, 337-371, 1986.

\bibitem{H} Harris, J.:  Curves in projective space. Les Press de l'Universit$\acute{e}$ de Montr$\acute{e}$al, 1982. 
 
\bibitem{HM} Harris, J., Morrison, I.:  Moduli of Curves. Graduate Text in Mathematics {\bf187}, Springer-Verlag New York, 1998.

\bibitem{I} Illusie, L.: Complex Cotangent et D$\acute{e}$formations I. Lecture Notes in Mathematics {\bf239}, Springer-Verlag, 1971

\bibitem{I1} Illusie, L.: Complex Cotangent et D$\acute{e}$formations II. Lecture Notes in Mathematics {\bf283}, Springer-Verlag, 1972

\bibitem{S} Sernesi, E.: Deformations of Algebraic Schemes. Springer Grundlehren {\bf334}, 2006.

\bibitem{V} Vistoli, A.:  The deformation theory of local complete intersections. http://arxiv.org/abs/alg-geom/9703008

\end{thebibliography}
\end{document}